\documentclass[dvipdfmx]{article}

\usepackage{amsmath,amssymb,amsthm,mathrsfs}
\usepackage[dvipdfmx]{graphicx}
\newcommand{\aki}{\hspace{2mm}}

\newtheorem{theo}{Theorem}[section]

\newtheorem{lemm}[theo]{Lemma}
\newtheorem{prop}[theo]{Proposition}

\newtheorem{remk}[theo]{Remark}

\title{The Beilinson Conjectures for CM Elliptic Curves via Hypergeometric Functions}
\author{Ryojun Ito\thanks{
Department of Mathematics and Informatics, Graduate School of Science, Chiba University, 
Yayoicho 1-33, Inage, Chiba, 263-8522 Japan. E-mail: afua9032@chiba-u.jp,  
2010 Mathematics Subject Classification: 11G05, 11S40, 19F27, 33C20.     
keywords: $L$-value of elliptic curve, regulator, Beilinson conjectures, hypergeometric function. }}

\begin{document}

\maketitle

\begin{abstract}
We consider certain CM elliptic curves which are related to Fermat curves, and express the values of $L$-functions at $s=2$ 
in terms of special values of generalized hypergeometric functions. 
We compare them and a similar result of Rogers-Zudilin with Otsubo's regulator formulas, 
and give a new proof of the Beilinson conjectures originally due to Bloch.
\end{abstract}

\section{Introduction}
The Beilinson conjectures \cite{bei1,bei2} are some very general statements extending
the class number formula which relates the values of $L$-functions at integers to regulators. 
For an elliptic curve $E$ over $\mathbb{Q}$, the conjecture concerning $L(E, 2)$ is originally due to Bloch \cite{blo1,blo2} 
and was proved by himself when $E$ has complex multiplication. 
The non-CM case follows from Beilinson's work on modular curves \cite{bei2} and the modularity of $E$ due to Wiles. 
The regulator map that we consider is given by
\begin{equation*}
r_{\mathscr{D}}: \aki H^{2}_{\mathscr{M}}(E, \mathbb{Q}(2))_{\mathbb{Z}} \longrightarrow 
H^{2}_{\mathscr{D}}(E_{\mathbb{R}}, \mathbb{R}(2)) 
\end{equation*} 
from the motivic cohomology to the Deligne cohomology (see Section \ref{regulator}).

Let $E_{N}$ be an elliptic curve over $\mathbb{Q}$ of conductor $N$. In this paper, we treat the cases $N=$ 27, 32 and 64, i.e. 
\begin{align*}
E_{27} : y^{2} &= x^{3}- \frac{27}{4}, \\
E_{32} : y^{2} &= x^{3}+ 4x, \\
E_{64} : y^{2} &= x^{3}- 4x .
\end{align*}
Note that $E_{27}$ is isogenous to the Fermat curve of degree $3$ 
and has complex multiplication by $\mathbb{Z}[(-1+\sqrt{-3})/2]$, 
and each of  $E_{32}$ and $E_{64}$ is a quotient of the Fermat curve of degree $4$ and has
complex multiplication by $\mathbb{Z}[\sqrt{-1}]$.

For Fermat curves, Ross \cite{ross1,ross2} constructed an element of the motivic cohomology group. 
Otsubo \cite{otsubo2,otsubo} expressed its regulator image in terms of special values of generalized hypergeometric functions $_{3}F_{2}$
\begin{align*}
{_{3}F_{2}}\left[ 
\left.
\begin{matrix}
a,b,c \\
e,f 
\end{matrix}
\right| z
\right]
&:=
\sum_{n=0}^{\infty} \frac{(a)_{n} (b)_{n} (c)_{n}}{(e)_{n} (f)_{n}}\frac{z^{n}}{n!} 
\end{align*}
where $(a)_{n} := \Gamma(a+n)/\Gamma(a)$ denotes the Pochhammer symbol 
(see Theorems \ref{n=3noreg}, \ref{n=4noregfor} and \ref{n=4noregfor2}).

By comparing the regulator image of Bloch's element with that of Ross' element, Otsubo \cite{otsubo} expressed 
the values $L^{\prime}(E_{27}, 0)$ and $L^{\prime}(E_{32}, 0)$ in terms of values of $_{3}F_{2}$ at $z=1$. Note that we have the 
functional equation (cf.\cite{d.w.})
\begin{equation}
L^{\prime}(E_{N}, 0) = \pm\frac{N}{(2\pi)^{2}} L(E_{N}, 2).  \label{functional}
\end{equation} 
On the other hand, in \cite{r.z}, 
Rogers and Zudilin expressed the value $L(E_{27}, 2)$ in terms of values of $_{3}F_{2}$ at $z=1$ directly by an analytic 
method (see Theorem \ref{n=3nol}).

The first purpose of this paper is to prove the following formulas. \\
{\bf Theorem} (Theorems \ref{n=4nol} and \ref{cond64})
\begin{align*}
L(E_{32}, 2) &= 
\frac{\sqrt{\pi}\Gamma^{2}\left( \frac{1}{4}\right)}{32\sqrt{2}}
{_{3}F_{2}}\left[ 
\left.
\begin{matrix}
\frac{1}{2}, \frac{1}{2}, 1 \\
\frac{3}{2}, \frac{3}{4}
\end{matrix}
\right| 1
\right]
-
\frac{\sqrt{\pi}\Gamma^{2}\left( \frac{3}{4}\right)}{8\sqrt{2}}
{_{3}F_{2}}\left[
\left.
\begin{matrix}
\frac{1}{2}, \frac{1}{2}, 1 \\
\frac{3}{2}, \frac{5}{4}
\end{matrix}
\right| 1
\right] , \\
L(E_{64}, 2) &= 
\frac{\sqrt{\pi}\Gamma^{2}\left( \frac{1}{4}\right)}{32}
{_{3}F_{2}}\left[ 
\left.
\begin{matrix}
\frac{1}{4}, \frac{1}{4}, 1 \\
\frac{1}{2}, \frac{5}{4}
\end{matrix}
\right| 1
\right]
-
\frac{\sqrt{\pi}\Gamma^{2}\left( \frac{3}{4}\right)}{48}
{_{3}F_{2}}\left[ 
\left.
\begin{matrix}
\frac{3}{4}, \frac{3}{4}, 1 \\
\frac{3}{2}, \frac{7}{4}
\end{matrix}
\right| 1
\right] .
\end{align*}

Rogers \cite[p.4036, (46)]{rog1} expressed the value $L(E_{32},2)$ in terms of a value of $_{3}F_{2}$ at $z = 1/2$. 
Zudilin \cite[p.391, Theorem 3]{zud2} expressed $L(E_{32}, 2)$ as a sum of values of $_{3}F_{2}$ at $z=1$.  
Our new representation of the value $L(E_{32}, 2)$, however, is appropriate for the comparison with regulators.

To prove the formulas above, we follow an analogous method to that of Rogers and Zudilin \cite{r.z}. 
The modularity theorem shows that the $L$-function of an elliptic curve is equal to the Mellin transform of a weight-two modular form. 
We know that the modular form corresponding to $E_{32}$ (resp. $E_{64}$) 
is $\eta^{2}(q^{4})\eta^{2}(q^{8})$ (resp. $\frac{\eta^{8}(q^{8})}{\eta^{2}(q^{4})\eta^{2}(q^{16})}$) (cf.\cite{onoken}), 
where $\eta (q)$ is the Dedekind eta function. Hence we have
\begin{align}
L(E_{32}, 2) &= - \int_{0}^{1} \eta^{2}(q^{4})\eta^{2}(q^{8})\log q \frac{dq}{q},  \label{e32}\\
L(E_{64}, 2) &= - \int_{0}^{1} \frac{\eta^{8}(q^{8})}{\eta^{2}(q^{4})\eta^{2}(q^{16})} \log q \frac{dq}{q} \label{e64}.
\end{align}
By Jacobi's triple product formula and Jacobi's imaginary transformation formula, 
each integral can be expressed as an integral of a product of Jacobi's theta functions. 
Then certain transformation reduces each of $L(E_{32}, 2)$ and $L(E_{64}, 2)$ to an integral of elementary functions.

The second purpose of this paper is to compare the regulators with the values $L^{\prime}(E_{N}, 0)$ 
via hypergeometric functions. 
Let $e_{E_{N}} \in H^{2}_{\mathscr{M}}(E_{N}, \mathbb{Q}(2))_{\mathbb{Z}}$ be an element which is constructed by mapping Ross' element, 
$\omega_{E_{N}}$ be the normalized real holomorphic differential form on $E_{N}$ 
and $\Omega_{\mathbb{R}}$ be its real period (see Section \ref{regulator}).
By comparing the representations of the regulators due to Otsubo \cite{otsubo}  
with the representations of the values of $L$-functions at $s=2$ explained above, 
we prove the Beilinson conjectures for $E_{27}$, $E_{32}$ and $E_{64}$. \\
{\bf Theorem} (Theorems \ref{comp1}, \ref{comp2} and \ref{comp3})
\begin{align*}
r_{\mathscr{D}}(e_{E_{27}}) 
&= - \frac{3}{2}L^{\prime}(E_{27}, 0) \Omega_{\mathbb{R}}(\omega_{E_{27}} - \overline{\omega_{E_{27}}}), \\
r_{\mathscr{D}}(e_{E_{32}}) 
&= - \frac{1}{2} L^{\prime}(E_{32},0)\Omega_{\mathbb{R}}(\omega_{E_{32}} - \overline{\omega_{E_{32}}}), \\
r_{\mathscr{D}}(e_{E_{64}}) 
&= - \frac{1}{2} L^{\prime}(E_{64},0)\Omega_{\mathbb{R}}(\omega_{E_{64}} - \overline{\omega_{E_{64}}}) .
\end{align*}
Note that, for $E_{27}$ and $E_{32}$, Otsubo \cite[Propositions 5.1 and 5.4]{otsubo} compared $r_{\mathscr{D}}(e_{E_{N}})$ with
the regulator image of Bloch's element. Hence the first two formulas can also be obtained by Bloch's result.
On the other hand, for $E_{64}$, $r_{\mathscr{D}}(e_{E_{64}})$ 
and the regulator image of Bloch's element has not been compared. 
Our result gives a rigorous proof of the formula which was found numerically in \cite{otsubo3}.

By Martin and Ono \cite{onoken}, those CM elliptic curves whose corresponding cusp form is an eta quotient 
are $E_{27}$, $E_{32}$, $E_{36}$, $E_{64}$ and $E_{144}$. The remaining cases are  
\begin{align*}
E_{36} : \aki y^{2} &= x^{3} + 1, \\
E_{144} : \aki y^{2} &= x^{3} - 1. 
\end{align*}
Note that each of $E_{36}$ and $E_{144}$ is a quotient of the Fermat curve of degree $6$.
Otsubo \cite{otsubo} expressed $r_{\mathscr{D}}(e_{E_{36}})$ and $r_{\mathscr{D}}(e_{E_{144}})$ in terms of values of $_{3}F_{2}$ at 
$z=1$. 
Hence a similar study of $L(E_{36}, 2)$ and $L(E_{144}, 2)$ would lead to a proof of the Beilinson conjectures. 
Further, for more general CM abelian varieties, one might be able to approach the Beilinson conjectures 
via hypergeometric functions.

The structure of this paper is as follows. 
In Section \ref{l-formulas}, we express the values $L(E_{32}, 2)$ and $L(E_{64}, 2)$ in terms of values of $_{3}F_{2}$ at $z=1$. 
Their proofs are analogous to the method of Rogers and Zudilin.
In Section \ref{comparison}, we stablish the relationship between $L^{\prime}(E_{27}, 0)$, $L^{\prime}(E_{32}, 0)$ 
and $L^{\prime}(E_{64}, 0)$ and the regulators via hypergeometric functions and in this way prove our main results.

\section{$L$-values}\label{l-formulas}
In this section, we express the values $L(E_{32}, 2)$ and $L(E_{64}, 2)$ in terms of values of hypergeometric functions.

\subsection{Conductor 32}\label{conductor 32}
First we prove the following formula.
\begin{prop}\label{thetaprod}
\begin{equation*}
L(E_{32}, 2) =
\frac{\pi}{32}\int_{0}^{1}
\theta_{2}(q)\theta_{3}(q)\left(\theta_{3}^{2}(q) - \theta_{2}^{2}(q) \right)
\log \left( \frac{\theta_{3}(q^{2})}{\theta_{2}(q^{2})}\right)  \frac{dq}{q}.
\end{equation*}
\end{prop}

We know that the modular form corresponding to $E_{32}$ is $\eta^{2}(q^{4})\eta^{2}(q^{8})$ (cf.\cite[p.3173, Theorem 2]{onoken}).
To show the formula above, we express this eta product as a product of Jacobi's theta functions.

\begin{lemm}\label{etatotheta}
\begin{equation*}
\eta^{2}(q^{4})\eta^{2}(q^{8}) = \frac{1}{4}\theta_{2}^{2}(q^{2})\theta_{4}^{2}(q^{4})
\end{equation*}
\end{lemm}
\begin{proof}
By Jacobi's triple product formula, we have (cf.\cite[p.3174]{onoken})
\begin{equation*}
\eta^{3}(q^{8}) = \sum_{n=0}^{\infty} (-1)^{n}(2n+1)q^{(2n+1)^{2}}, \hspace{4mm}
\frac{\eta^{2}(q)}{\eta(q^{2})} = \sum_{n=-\infty}^{\infty} (-1)^{n}q^{n^{2}} = \theta_{4}(q).
\end{equation*}
If we use the following formula \cite[p.68, Proposition 3.1]{borwein}, \cite[p.40, Entry 25]{ramanujan3}
\begin{equation*}
2\sum_{n=0}^{\infty} (-1)^{n}(2n+1)q^{(2n+1)^{2}} = \theta_{2}(q^{4})\theta_{3}(q^{4})\theta_{4}(q^{4}), \aki
2\theta_{2}(q^{2})\theta_{3}(q^{2})=\theta_{2}^{2}(q),
\end{equation*}
then we obtain the lemma.
\end{proof}

By \eqref{e32} and Lemma \ref{etatotheta}, we have
\begin{align*}
L(E_{32}, 2) = - \frac{1}{4}\int_{0}^{1} \theta_{2}^{2}(q^{2})\theta_{4}^{2}(q^{4}) \log q \frac{dq}{q} .
\end{align*}
By substituting $q^{2} \mapsto q$ and setting $q = e^{-2\pi u}$, we obtain
\begin{equation*}
L(E_{32}, 2) = \frac{\pi^{2}}{4}\int_{0}^{\infty}\theta_{2}^{2}(e^{-2\pi u})\cdot u\theta_{4}^{2}(e^{-4\pi u})du.
\end{equation*}
We use the following Lambert series expansion
\begin{equation*}
\theta_{2}^{2}(q) =4\sum_{n=0}^{\infty} \frac{q^{n+1/2}}{1+q^{2n+1}} =4 \sum_{n=1}^{\infty} 
\frac{\chi_{-4}(n)q^{n/2}}{1-q^{n}} 
=4\sum_{n,k=1}^{\infty} \chi_{-4}(n)q^{n(k-1/2)}
\end{equation*}
where $\chi_{-4} (n) := {\rm Im} ( i^{n})$. This follows from the following formulas \cite[p.115, (8.2)]{ramanujan3},  
\cite[p.35, (2.1.8)]{borwein}
\begin{equation*}
\theta_{3}^{2} (q) =1 + 4\sum_{n=1}^{\infty} \frac{q^{n}}{1+q^{2n}}, \hspace{7mm}
\theta_{2}^{2}(q^{2}) = \theta_{3}^{2}(q) - \theta_{3}^{2}(q^{2}) .
\end{equation*}
By Jacobi's imaginary transformation formula \cite[p.40, (2.3.3)]{borwein}, we have 
\begin{equation*}
u\theta_{4}^{2}(e^{-4\pi u}) = \frac{1}{4} \theta_{2}^{2}(e^{-\frac{\pi}{4u}}) 
= \sum_{r,s=1}^{\infty} \chi_{-4}(r)e^{-\frac{\pi r(s-1/2)}{4u}}.
\end{equation*}
Therefore we obtain
\begin{align*}
L(E_{32}, 2) = \pi^{2} \int_{0}^{\infty} \sum_{n,k,r,s=1}^{\infty} \chi_{-4}(nr) e^{-2\pi un(k-1/2)} 
\cdot e^{-\frac{\pi r(s-1/2)}{4u}} du. 
\end{align*}
By substituting $u \mapsto ru/(k-1/2)$, we have 
\begin{equation}
L(E_{32}, 2) 
=\pi^{2}\int_{0}^{\infty} \left( \sum_{n,r=1}^{\infty}r\chi_{-4}(nr)e^{-2\pi unr}\right) 
\left( \sum_{k,s=1}^{\infty}\frac{e^{-\frac{\pi}{4u}(s-1/2)(k-1/2)}}{k-\frac{1}{2}} \right) du . \label{totyuu}
\end{equation}
We compute the two series in the integral in the following lemmas.

\begin{lemm}\label{seriescal1}
\begin{equation*}
\sum_{k,s=1}^{\infty}\frac{e^{-\frac{\pi}{4u}(s-1/2)(k-1/2)}}{k-\frac{1}{2}}
= \frac{1}{2} \log \frac{\theta_{3}(q^{8})}{\theta_{2}(q^{8})}
\end{equation*}
where $q = e^{-2\pi u}$.
\end{lemm}
\begin{proof}
We have
\begin{align*}
&\sum_{k,s=1}^{\infty}\frac{e^{-\frac{\pi}{4u}(s-1/2)(k-1/2)}}{k-\frac{1}{2}}
=\log \prod_{s\geqq 1}\left| \frac{1+e^{-\frac{\pi (2s-1)}{16u}}}{1-e^{-\frac{\pi (2s-1)}{16u}}} \right| \\
&=\log \prod_{s\geqq 1}\left| 
\frac{\left( 1-e^{-\frac{\pi s}{8u}} \right)^{3}}{\left( 1-e^{-\frac{\pi s}{4u}} \right) \left( 1-e^{-\frac{\pi s}{16u}} \right)^{2}}
\right|
=\log 
\left|
\frac{\eta^{3}(e^{-\frac{\pi}{8u}})}{\eta(e^{-\frac{\pi}{4u}})\eta^{2}(e^{-\frac{\pi}{16u}})}
\right|.
\end{align*}
Now, if we apply the involution for the eta function
\begin{equation*}
\eta(e^{\frac{-2\pi i}{\tau}}) = \sqrt{-i \tau} \eta(e^{2\pi i \tau}),
\end{equation*}
then we obtain
\begin{align*}
\frac{\eta^{3}(e^{-\frac{\pi}{8u}})}{\eta(e^{-\frac{\pi}{4u}})\eta^{2}(e^{-\frac{\pi}{16u}})}
=\frac{\eta^{3}(e^{-32\pi u})}{\sqrt{2} \eta(e^{-16\pi u}) \eta^{2}(e^{-64\pi u})}
=\frac{\eta^{3}(q^{16})}{\sqrt{2} \eta(q^{8}) \eta^{2}(q^{32})}
\end{align*}
where we set $q = e^{-2\pi u}$. 
If we use the formulas \cite[p.3174]{onoken}
\begin{align}
\frac{\eta^{5}(q^{2})}{\eta^{2}(q)\eta^{2}(q^{4})} = \theta_{3}(q), \aki  
\frac{\eta^{2}(q)}{\eta(q^{2})}=\theta_{4}(q),  \label{jacobifor1}
\end{align}
and 
\begin{align}
\eta^{3}(q^{8}) = \frac{1}{2}\theta_{2}(q^{4})\theta_{3}(q^{4})\theta_{4}(q^{4}),\label{jacobifor2}
\end{align}
then 
\begin{align*}
\frac{\eta^{3}(q^{16})}{\sqrt{2}\eta(q^{8}) \eta^{2}(q^{32})}
&= \frac{\eta^{5}(q^{16})}{\eta^{2}(q^{8}) \eta^{2}(q^{32})} \cdot \frac{\eta(q^{8})\eta(q^{16})}{\sqrt{2}\eta^{3}(q^{16})} \\
&=\theta_{3}(q^{8}) \cdot 
\frac{\frac{1}{\sqrt{2}}\theta_{2}^{\frac{1}{2}}(q^{8})\theta_{3}^{\frac{1}{2}}(q^{8})\theta_{4}(q^{8})}
{\sqrt{2}\cdot \frac{1}{2}\theta_{2}(q^{8})\theta_{3}(q^{8})\theta_{4}(q^{8})} 
=\left( \frac{\theta_{3}(q^{8})}{\theta_{2}(q^{8})}\right) ^{\frac{1}{2}}.
\end{align*}
Hence we have the lemma.
\end{proof}

\begin{lemm}\label{seriescal2}
\begin{equation*}
\sum_{n,r=1}^{\infty}r\chi_{-4}(nr)q^{nr} 
= \frac{1}{2} \theta_{2}(q^{4})\theta_{3}(q^{4})\left(\theta_{3}^{2}(q^{4}) - \theta_{2}^{2}(q^{4}) \right).
\end{equation*}
\end{lemm}
\begin{proof}
Since $\chi_{-4} (n) = {\rm Im}( i^{n} ) $, 
\begin{align*}
\sum_{n,r=1}^{\infty}r\chi_{-4}(nr)q^{nr} = {\rm Im}\left( \sum_{n,r \geqq 1} r \left( i q \right) ^{nr} \right)
= {\rm Im}
\left( \sum_{r \geqq 1}  \frac{r \left( i q \right) ^{r}}{1-\left( i q \right) ^{r}} \right)
= -\frac{1}{24}{\rm Im}\left( L(iq)\right)
\end{align*}
where 
\begin{equation*}
L(q) := 1 -24\sum_{n=1}^{\infty} \frac{nq^{n}}{1-q^{n}}.
\end{equation*}
Ramanujan proved \cite[p.377, Entry 38]{ramanujan4} that
\begin{equation*}
3\theta_{3}^{4}(q) = 4L(q^{4}) - L(q),
\end{equation*}
hence we have
\begin{align*}
\sum_{n,r=1}^{\infty}r \chi_{-4}(nr)q^{nr} = \frac{1}{8}{\rm Im}\left( \theta_{3}^{4}(iq) \right).
\end{align*}
If we use the following formula \cite[p.73]{borwein}
\begin{equation*}
\theta_{3}(iq) = \theta_{3}(q^{4}) + i\theta_{2}(q^{4}),
\end{equation*}
then we obtain
\begin{align*}
{\rm Im}\left( \theta_{3}^{4}(iq) \right) 
= 4\theta_{3}^{3}(q^{4})\theta_{2}(q^{4}) - 4\theta_{3}(q^{4})\theta_{2}^{3}(q^{4}).
\end{align*}
Therefore we have the lemma.
\end{proof}

By applying Lemmas \ref{seriescal1} and \ref{seriescal2} to \eqref{totyuu}, we obtain
\begin{equation*}
L(E_{32}, 2)=
\frac{\pi}{8}\int_{0}^{1} \theta_{2}(q^{4})\theta_{3}(q^{4})\left(\theta_{3}^{2}(q^{4}) - \theta_{2}^{2}(q^{4}) \right)
\left( \log \frac{\theta_{3}(q^{8})}{\theta_{2}(q^{8})} \right) \frac{dq}{q}. 
\end{equation*}
Then, by substituting $q^{4} \mapsto q$, we have Proposition \ref{thetaprod}. $\qed$ \\

Now we express the value $L(E_{32}, 2)$ in terms of the values of $_{3}F_{2}$ at $z=1$.
\begin{theo}\label{n=4nol}
\begin{equation*}
L(E_{32}, 2) = 
\frac{\sqrt{\pi}\Gamma^{2}\left( \frac{1}{4}\right)}{32\sqrt{2}}
{_{3}F_{2}}\left[ 
\left.
\begin{matrix}
\frac{1}{2}, \frac{1}{2}, 1 \\
\frac{3}{2}, \frac{3}{4}
\end{matrix}
\right| 1
\right]
-
\frac{\sqrt{\pi}\Gamma^{2}\left( \frac{3}{4}\right)}{8\sqrt{2}}
{_{3}F_{2}}\left[ 
\left.
\begin{matrix}
\frac{1}{2}, \frac{1}{2}, 1 \\
\frac{3}{2}, \frac{5}{4}
\end{matrix}
\right| 1
\right] .
\end{equation*}
\end{theo}
\begin{proof}
Let 
\begin{equation*}
z(x) := {_{2}}F_{1}
\left[ \left.
\begin{matrix}
\frac{1}{2}, \frac{1}{2} \\
1 
\end{matrix}
\right| x
\right], \hspace{5mm} y(x) := \pi\frac{z(1-x)}{z(x)}.
\end{equation*}
Set 
\begin{equation*}
q = e^{-y(x)} .
\end{equation*} 
We know \cite[p.87, Entry 30]{ramanujan2}, \cite[p.101, Entry 6]{ramanujan3} that
\begin{equation*}
\theta_{3}^{4}(q)\frac{dq}{q} = \frac{dx}{x(1-x)}.
\end{equation*} 
It is also known \cite[Entry 10, 11]{ramanujan3} that 
\begin{align*}
\theta_{3}(q) = \sqrt{z(x)}, &\hspace{5mm} \theta_{3}(q^{2}) = \sqrt{\frac{z(x)}{2}}\left( 1+\sqrt{1-x} \right) ^{\frac{1}{2}}, \\
\theta_{2}(q) = \sqrt{z(x)}x^{\frac{1}{4}}, &\hspace{5mm} \theta_{2}(q^{2}) = \sqrt{\frac{z(x)}{2}}(1- \sqrt{1-x})^{\frac{1}{2}}.
\end{align*}
Therefore, 
\begin{align*}
L(E_{32}, 2) 
&=-\frac{\pi}{32}\int_{0}^{1}
\frac{\theta_{2}(q)}{\theta_{3}(q)}\left(1 - \frac{\theta_{2}^{2}(q)}{\theta_{3}^{2}(q)} \right)
\log \left( \frac{\theta_{2}(q^{2})}{\theta_{3}(q^{2})}\right)  \theta_{3}^{4}(q)\frac{dq}{q} \\
&=-\frac{\pi}{32} \int_{0}^{1}
x^{1/4}\left(1 - x^{1/2} \right)
\log \left( \frac{(1- \sqrt{1-x})^{\frac{1}{2}}}{ (1+\sqrt{1-x}) ^{\frac{1}{2}}} \right)  \frac{dx}{x(1-x)} \\
&=-\frac{\pi}{32}\int_{0}^{1}
x^{1/4}\left(1 - x^{1/2} \right)
\log \left( \frac{1- (1-x)^{1/2}}{ x^{1/2}}\right)  \frac{dx}{x(1-x)} .
\end{align*}
If we use the formula
\begin{equation*}
\log \left( \frac{1- (1-x)^{1/2}}{ x^{1/2}}\right)
= \sum_{n=1}^{\infty} \frac{(1-x)^{n}-2(1-x)^{\frac{n}{2}}}{2n},
\end{equation*}
and perform term-by-term integration using beta integrals
\begin{equation*}
B(\alpha, \beta) := \int_{0}^{1} x^{\alpha -1}(1-x)^{\beta -1} dx, 
\end{equation*}
we obtain
\begin{align*}
L(E_{32},2) 
=-\frac{\pi}{32} \sum_{n=1}^{\infty}\frac{1}{2n}\left( B\left(\frac{1}{4},n\right) - 2B\left(\frac{1}{4},\frac{n}{2}\right)
- B\left(\frac{3}{4},n \right) + 2 B\left(\frac{3}{4}, \frac{n}{2} \right)   \right). 
\end{align*}
Using Pochhammer symbols, the values of beta function are represented as follows:
\begin{align*}
B\left( \frac{1}{4}, n \right) &= \frac{\Gamma\left( \frac{1}{4}\right)\Gamma (n)}{\Gamma\left( \frac{1}{4}+n \right)} 
= \frac{\Gamma(n)}{\left( \frac{1}{4} \right)_{n}} ,\hspace{5mm} 
B\left( \frac{3}{4}, n \right) = \frac{\Gamma\left( \frac{3}{4}\right)\Gamma (n)}{\Gamma\left( \frac{3}{4}+n \right)} 
= \frac{\Gamma(n)}{\left( \frac{3}{4} \right)_{n}}, \\
B\left( \frac{1}{4}, \frac{n}{2} \right) &= \left\{
\begin{aligned}
&B\left( \frac{1}{4}, \frac{2m}{2} \right) = \frac{\Gamma(m)}{\left( \frac{1}{4} \right)_{m}} &(n \equiv 0 \mod 2), \\
&B\left( \frac{1}{4}, \frac{2m+1}{2} \right) 
= \frac{\Gamma\left( \frac{1}{4}\right)\Gamma\left( \frac{1}{2}\right)\left( \frac{1}{2} \right)_{m}}
{\Gamma\left( \frac{3}{4} \right)\left( \frac{3}{4} \right)_{m}} &(n \equiv 1 \mod 2),
\end{aligned}
\right. \\
B\left( \frac{3}{4}, \frac{n}{2} \right) &= \left\{
\begin{aligned}
&B\left( \frac{3}{4}, \frac{2m}{2} \right) = \frac{\Gamma(m)}{\left( \frac{3}{4} \right)_{m}} &(n \equiv 0 \mod 2), \\
&B\left( \frac{3}{4}, \frac{2m+1}{2} \right) 
= \frac{4\Gamma\left( \frac{3}{4}\right)\Gamma\left( \frac{1}{2}\right)\left( \frac{1}{2}\right)_{m}}
{\Gamma\left( \frac{1}{4}\right)\left( \frac{5}{4}\right)_{m}} &(n \equiv 1 \mod 2).
\end{aligned}
\right.
\end{align*}
Hence we obtain
\begin{align*}
L(E_{32}, 2) = \frac{\pi\Gamma\left( \frac{1}{4}\right)\Gamma\left( \frac{1}{2}\right)}{32\Gamma\left( \frac{3}{4}\right)} 
\sum_{m=0}^{\infty} \frac{\left( \frac{1}{2} \right)_{m}}{(2m+1)\left( \frac{3}{4} \right)_{m}}
- \frac{\pi\Gamma\left( \frac{3}{4}\right)\Gamma\left( \frac{1}{2}\right)}{8\Gamma\left( \frac{1}{4}\right)} 
\sum_{m=0}^{\infty} \frac{\left( \frac{1}{2} \right)_{m}}{(2m+1)\left( \frac{5}{4} \right)_{m}} .
\end{align*}
Note that these series are hypergeometric functions. In fact, we have
\begin{align*}
\sum_{m=0}^{\infty} \frac{\left( \frac{1}{2} \right)_{m}}{(2m+1)\left( \frac{3}{4} \right)_{m}}
= \sum_{m=0}^{\infty} \frac{\left( \frac{1}{2} \right)_{m}\left( \frac{1}{2} \right)_{m}(1)_{m}}
{(2m+1)\left( \frac{1}{2} \right)_{m}\left( \frac{3}{4} \right)_{m}(1)_{m}} 
={_{3}F_{2}}\left[ 
\left.
\begin{matrix}
\frac{1}{2}, \frac{1}{2}, 1 \\
\frac{3}{2}, \frac{3}{4}
\end{matrix}
\right| 1
\right]. \label{n=4nobaainol1}
\end{align*}
Similarly, we have 
\begin{align*}
\sum_{m=0}^{\infty} \frac{\left( \frac{1}{2} \right)_{m}}{(2m+1)\left( \frac{5}{4} \right)_{m}}
={_{3}F_{2}}\left[ 
\left.
\begin{matrix}
\frac{1}{2}, \frac{1}{2}, 1 \\
\frac{3}{2}, \frac{5}{4}
\end{matrix}
\right| 1
\right]. 
\end{align*}
Hence we obtain
\begin{align*}
L(E_{32}, 2) = 
\frac{\pi\Gamma\left( \frac{1}{4}\right)\Gamma\left( \frac{1}{2}\right)}{32\Gamma\left( \frac{3}{4} \right)}
{_{3}F_{2}}\left[ 
\left.
\begin{matrix}
\frac{1}{2}, \frac{1}{2}, 1 \\
\frac{3}{2}, \frac{3}{4}
\end{matrix}
\right| 1
\right]
-
\frac{\pi\Gamma\left( \frac{3}{4}\right)\Gamma\left( \frac{1}{2}\right)}{8\Gamma\left( \frac{1}{4} \right)}
{_{3}F_{2}}\left[ 
\left.
\begin{matrix}
\frac{1}{2}, \frac{1}{2}, 1 \\
\frac{3}{2}, \frac{5}{4}
\end{matrix}
\right| 1
\right] .
\end{align*}
If we use the formulas
\begin{equation*}
\Gamma\left( \frac{1}{4}\right) \Gamma\left( \frac{3}{4}\right) = \sqrt{2}\pi, \hspace{5mm} \Gamma\left( \frac{1}{2}\right)=\sqrt{\pi},
\end{equation*}
then we finally obtain the theorem.
\end{proof}

\subsection{Conductor 64}

We know that the modular form corresponding to $E_{64}$ is $\frac{\eta^{8}(q^{8})}{\eta^{2}(q^{4})\eta^{2}(q^{16})}$ (cf.\cite{onoken}).

\begin{prop}\label{lnothetadenohyouji2}
\begin{equation*}
L(E_{64}, 2) =
\frac{\pi}{64}\int_{0}^{1}
\theta_{2}(q)\theta_{3}(q)\left(\theta_{3}^{2}(q) - \theta_{2}^{2}(q) \right)
\log \left( \frac{\theta_{3}(q^{4})}{\theta_{2}(q^{4})}\right)  \frac{dq}{q}.
\end{equation*}
\end{prop}
\begin{proof}
We have the identity 
\begin{equation*}
\frac{\eta^{8}(q^{8})}{\eta^{2}(q^{4})\eta^{2}(q^{16})} = \frac{1}{4}\theta_{2}^{2}(q^{2})\theta_{4}^{2}(q^{8}).
\end{equation*}
This identity follows from \eqref{jacobifor1}, \eqref{jacobifor2} and $\theta_{3}(q) \theta_{4}(q) = \theta_{4}^{2}(q^{2})$
\cite[p.34, (2.1.7ii)]{borwein}.
Therefore we obtain
\begin{equation*}
L(E_{64}, 2) = -\frac{1}{4}\int_{0}^{1} \theta_{2}^{2}(q^{2})\theta_{4}^{2}(q^{8}) \log q \frac{dq}{q}
= -\frac{1}{16}\int_{0}^{1} \theta_{2}^{2}(q)\theta_{4}^{2}(q^{4}) \log q \frac{dq}{q}.
\end{equation*}
By similar calculations as in Proposition \ref{thetaprod}, we have the proposition. 
\end{proof}

Now we express the value $L(E_{64}, 2)$ in terms of values of $_{3}F_{2}$ at $z=1$.
\begin{theo}\label{cond64}
\begin{equation*}
L(E_{64}, 2) = 
\frac{\sqrt{\pi}\Gamma^{2}\left( \frac{1}{4}\right)}{32}
{_{3}F_{2}}\left[ 
\left.
\begin{matrix}
\frac{1}{4}, \frac{1}{4}, 1 \\
\frac{1}{2}, \frac{5}{4}
\end{matrix}
\right| 1
\right]
-
\frac{\sqrt{\pi}\Gamma^{2}\left( \frac{3}{4}\right)}{48}
{_{3}F_{2}}\left[ 
\left.
\begin{matrix}
\frac{3}{4}, \frac{3}{4}, 1 \\
\frac{3}{2}, \frac{7}{4}
\end{matrix}
\right| 1
\right] .
\end{equation*}
\end{theo}
\begin{proof}
If we set 
\begin{equation*}
q = e^{-y(x)}, 
\end{equation*} 
then we know \cite[Entry 10, 11]{ramanujan3}
\begin{align*}
\theta_{3}(q) = \sqrt{z(x)}, &\hspace{5mm} \theta_{3}(q^{4}) = \frac{1}{2}\sqrt{z(x)}(1+(1-x)^{1/4}), \\
\theta_{2}(q) = \sqrt{z(x)}x^{\frac{1}{4}}, &\hspace{5mm} \theta_{2}(q^{4}) = \frac{1}{2}\sqrt{z(x)}(1-(1-x)^{1/4}).
\end{align*}
Therefore, 
\begin{align*}
L(E_{64}, 2) 
&=\frac{\pi}{64}\int_{0}^{1}
\frac{\theta_{2}(q)}{\theta_{3}(q)}\left(1 - \frac{\theta_{2}^{2}(q)}{\theta_{3}^{2}(q)} \right)
\log \left( \frac{\theta_{2}(q^{4})}{\theta_{3}(q^{4})}\right)  \theta_{3}^{4}(q)\frac{dq}{q} \\
&=\frac{\pi}{64} \int_{0}^{1}
x^{1/4}\left(1 - x^{1/2} \right)
\log \left( \frac{1+(1-x)^{1/4}}{1-(1-x)^{1/4}} \right)  \frac{dx}{x(1-x)}.
\end{align*}
If we use the following formula
\begin{equation*}
\log \left( \frac{1+(1-x)^{1/4}}{1-(1-x)^{1/4}} \right)
= -2\sum_{n=1}^{\infty} \frac{(1-x)^{n/2}- 2(1-x)^{n/4}}{2n},
\end{equation*}
then 
\begin{align*}
L(E_{64},2) 
=-\frac{\pi}{32} \sum_{n=1}^{\infty}\frac{1}{2n}\left( B\left(\frac{1}{4},\frac{n}{2}\right) - 2B\left(\frac{1}{4},\frac{n}{4}\right)
- B\left(\frac{3}{4},\frac{n}{2} \right) + 2 B\left(\frac{3}{4}, \frac{n}{4} \right)   \right).
\end{align*}
By similar calculations as in Theorem \ref{n=4nol}, we obtain the theorem.
\end{proof}

\begin{remk}
{\rm After writing this paper, 
the author learned that the same formula was obtained independently in the unpublished notes of Rogers \cite[Theorem 5]{rog2}.}
\end{remk}

\section{Comparisons}\label{comparison}
In this section, we compare the regulators with the $L$-values for $E_{27}$, $E_{32}$ and $E_{64}$ via hypergeometric functions.

\subsection{Regulator of Curves}\label{regulator}
Here we recall the Beilinson regulator map for curves (cf.\cite{sch}). 

Let $C$ be a projective smooth curve over $\mathbb{Q}$. 
The {\it regulator map} $r_{\mathscr{D}}$ defined by Beilinson is a canonical map 
from the {\it integral part of the motivic cohomology group} $H^{2}_{\mathscr{M}}(C, \mathbb{Q}(2))_{\mathbb{Z}}$ 
to the {\it real Deligne cohomology group} $H^{2}_{\mathscr{D}}(C_{\mathbb{R}}, \mathbb{R}(2))$ (cf.\cite{sch})
\begin{equation*}
r_{\mathscr{D}}: \aki H^{2}_{\mathscr{M}}(C, \mathbb{Q}(2))_{\mathbb{Z}} \longrightarrow 
H^{2}_{\mathscr{D}}(C_{\mathbb{R}}, \mathbb{R}(2)).
\end{equation*}

We have an isomorphism (cf.\cite{jan})
\begin{equation*}
H^{2}_{\mathscr{M}}(C, \mathbb{Q}(2)) \cong 
{\rm Ker}\left( \tau \otimes \mathbb{Q} : K_{2}^{M}(\mathbb{Q}(C)) \otimes \mathbb{Q} \longrightarrow \bigoplus_{x \in C^{(1)}} 
\kappa(x)^{*} \otimes \mathbb{Q} \right).
\end{equation*}  
Here $C^{(1)}$ is the set of closed points on $C$, $\kappa(x)$ is the residue field, and $\tau = (\tau_{x})$ is the {\it tame symbol}
on $K_{2}^{M}(\mathbb{Q}(C))$
\begin{equation*}
\tau_{x}(\{f, g \}) = (-1)^{{\rm ord}_{x}f{\rm ord}_{x}g} \left(\frac{f^{{\rm ord}_{x}g}}{g^{{\rm ord}_{x}f}} \right)(x).
\end{equation*}
The {\it integral part} $H^{2}_{\mathscr{M}}(C, \mathbb{Q}(2))_{\mathbb{Z}}$ is defined to be the image of the $K$-group of 
a regular model of $C$ proper and flat over $\mathbb{Z}$.

On the other hand, we have an isomorphism (cf.\cite{e.v})
\begin{equation*}
H^{2}_{\mathscr{D}}(C_{\mathbb{R}}, \mathbb{R}(2)) \cong H^{1}(C(\mathbb{C}), \mathbb{R}(1))^{+} .
\end{equation*}
Here $+$ denotes the part fixed by the {\it de Rham conjugation} $F_{\infty} \otimes c_{\infty}$, where the {\it infinite Frobenius}
$F_{\infty}$ is the complex conjugation acting on $C(\mathbb{C})$ and $c_{\infty}$ is the complex conjugation on the coefficients.

Let $E$ be an elliptic curve.
Let $\omega_{E} \in H^{0}(E(\mathbb{C}), \Omega ^{1})^{+}$  be the real holomorphic differential form normalized so that
\begin{equation*}
\frac{1}{2\pi \sqrt{-1}} \int_{E(\mathbb{C})} \omega_{E} \wedge \overline{\omega_{E}} = -1
\end{equation*}
where $\overline{\omega_{E}} := c_{\infty}\omega_{E} = F_{\infty}\omega_{E}$. 
Note that $H^{1}(E(\mathbb{C}), \mathbb{R}(1))^{+}$ is generated by $\omega_{E} - \overline{\omega_{E}}$. 
Let $E(\mathbb{R})^{0}$ be the connected component of the origin with the orientation such that the real period
\begin{equation*}
\Omega_{\mathbb{R}} = \int_{E(\mathbb{R})^{0}} \omega_{E}
\end{equation*}
is positive. Then the  Beilinson conjectures
read that there exists an element $e \in H_{\mathscr{M}}^{2}(E, \mathbb{Q})_{\mathbb{Z}}$ such that 
\begin{equation*}
r_{\mathscr{D}}(e) = L^{\prime}(E, 0) \Omega_{\mathbb{R}}(\omega_{E} - \overline{\omega_{E}}) .
\end{equation*}

Let $X_{n}$ be the Fermat curve of degree $n$ 
\begin{equation*}
X_{n} : \aki u^{n} + v^{n} = 1.
\end{equation*}
We have a finite map $f: X_{n} \longrightarrow E_{N}$ which is defined by
\begin{align*}
\begin{aligned}
&f(u,v) = \left( \frac{3v}{1-u}, \frac{9(1+u)}{2(1-u)} \right) &\mbox{for}\aki (n, N) = (3, 27), \\
&f(u,v) = \left( \frac{2(1-v^{2})}{u^{2}}, \frac{4(1-v^{2})}{u^{3}} \right) &\mbox{for}\aki (n, N) = (4, 32), \\
&f(u,v) = \left( \frac{2(u^{2}-1)}{v^{2}}, \frac{4u(u^{2}-1)}{v^{3}} \right) &\mbox{for}\aki (n, N) = (4, 64).
\end{aligned}
\end{align*}
Let $e_{n} := \{1-u, 1-v \} \in H^{2}_{\mathscr{M}}(X_{n}, \mathbb{Q}(2))_{\mathbb{Z}}$ be Ross' element \cite{ross2}, 
and set \cite{otsubo2}
\begin{equation*}
e_{E_{N}} := f_{*}(e_{n}) \in H^{2}_{\mathscr{M}}(E_{N}, \mathbb{Q}(2))_{\mathbb{Z}} .
\end{equation*}
Otsubo \cite{otsubo2, otsubo} expressed its regulator image in terms of values of hypergeometric functions $\tilde{F}$
\begin{equation*}
\tilde{F}(\alpha, \beta) := \left(\frac{\Gamma(\alpha)\Gamma(\beta)}{\Gamma(\alpha + \beta)} \right)^{2}
{_{3}}F_{2}
\left[ 
\left.
\begin{matrix}
\alpha, \beta, \alpha + \beta -1 \\
\alpha + \beta , \alpha + \beta 
\end{matrix}
\right| 1
\right].
\end{equation*}
This is monotonically decreasing with respect to each parameter \cite[Proposition 4.25]{otsubo2}.
If we use Thomae's formula \cite[p.14, (1)]{bailey}
\begin{equation*}
{_{3}F_{2}}\left[ 
\left.
\begin{matrix}
a,b,c  \\
e,f 
\end{matrix}
\right| 1
\right] = \frac{\Gamma(e)\Gamma(f)\Gamma(s)}{\Gamma(a)\Gamma(b+s)\Gamma(c+s)} 
{_{3}F_{2}}\left[ 
\left.
\begin{matrix}
e-a, f-a , s \\
s+c, s+b 
\end{matrix}
\right| 1
\right] 
\end{equation*}
where $s:= e+f - (a+b+c)$, we have
\begin{align}
\tilde{F}(\alpha, \beta) 
= 
\frac{\Gamma(\alpha)\Gamma(\beta)}{\beta\Gamma(\alpha+\beta)}
{_{3}}F_{2}
\left[ 
\left.
\begin{matrix}
\beta, \beta, 1 \\
\alpha + \beta , \beta +1
\end{matrix}
\right| 1
\right]. \label{dixon}
\end{align}

\subsection{Conductor 27}

Otsubo proved the following formula. 
See \cite{otsubo}, Section 5.2 for the relation of $\omega_{E_{27}}$ with a form on the Fermat curve. 
\begin{theo}[{\cite[Theorem 3.2]{otsubo}}]\label{n=3noreg}
With the notations as above, we have
\begin{equation*}
r_{\mathscr{D}}(e_{E_{27}}) = -\frac{1}{6}\sqrt{\frac{\sqrt{3}}{2\pi}}
\left( \tilde{F}\left( \frac{1}{3}, \frac{1}{3} \right) - \tilde{F}\left(\frac{2}{3} , \frac{2}{3}
\right) \right) \left( \omega_{E_{27}} - \overline{\omega_{E_{27}}}\right) .
\end{equation*}
and $\tilde{F}\left( \frac{1}{3}, \frac{1}{3} \right) - \tilde{F}\left(\frac{2}{3} , \frac{2}{3}
\right) \neq 0$.
\end{theo}

On the other hand, Rogers and Zudilin proved the following formula.
\begin{theo}[{\cite[Theorem 1]{r.z}}] \label{n=3nol}
\begin{equation*}
L(E_{27}, 2) =  
\frac{\Gamma^{3}\left(\frac{1}{3} \right)}{27} {_{3}F_{2}}\left[ 
\left.
\begin{matrix}
\frac{1}{3}, \frac{1}{3}, 1 \\
\frac{2}{3}, \frac{4}{3}
\end{matrix}
\right| 1
\right]
-
\frac{\Gamma^{3}\left(\frac{2}{3} \right)}{18} {_{3}F_{2}}\left[ 
\left.
\begin{matrix}
\frac{2}{3}, \frac{2}{3}, 1 \\
\frac{4}{3}, \frac{5}{3}
\end{matrix}
\right| 1
\right] .
\end{equation*}
\end{theo}

By comparing the formulas above, we prove the Beilinson conjectures for $E_{27}$.
\begin{theo}\label{comp1}
With the notations as above, we have
\begin{equation*}
r_{\mathscr{D}}(e_{E_{27}}) = - \frac{3}{2}L^{\prime}(E_{27}, 0) \Omega_{\mathbb{R}}(\omega_{E_{27}} - \overline{\omega_{E_{27}}}).
\end{equation*}
\end{theo}
\begin{proof}
By \eqref{dixon},  we have 
\begin{align*}
\tilde{F}\left( \frac{1}{3}, \frac{1}{3} \right)
&=
\frac{3\sqrt{3}}{2\pi}\Gamma^{3}\left(\frac{1}{3}\right)
{_{3}F_{2}}\left[ 
\left.
\begin{matrix}
\frac{1}{3}, \frac{1}{3}, 1 \\
\frac{2}{3}, \frac{4}{3}
\end{matrix}
\right| 1
\right], \\
\tilde{F}\left(\frac{2}{3} , \frac{2}{3}\right)
&
=\frac{9\sqrt{3}}{4\pi}\Gamma^{3}\left(\frac{2}{3}\right)
{_{3}F_{2}}\left[ 
\left.
\begin{matrix}
\frac{2}{3}, \frac{2}{3}, 1 \\
\frac{4}{3}, \frac{5}{3}
\end{matrix}
\right| 1
\right].
\end{align*}
Hence, by Theorems \ref{n=3noreg} and \ref{n=3nol}, we have
\begin{equation*}
r_{\mathscr{D}}(e_{E_{27}})
= -\frac{27}{4\pi}\sqrt{\frac{3\sqrt{3}}{2\pi}} L(E_{27}, 2)\left(\omega_{E_{27}} - \overline{\omega_{E_{27}}} \right) .
\end{equation*}
By the functional equation \eqref{functional} and the fact that the root number (the sign of the functional equation) is $1$, we have
\begin{align*}
r_{\mathscr{D}}(e_{E_{27}})
= -\sqrt{\frac{3\pi\sqrt{3}}{2}} L^{\prime}(E_{27}, 0)\left( \omega_{E_{27}} - \overline{\omega_{E_{27}}} \right) .
\end{align*}
We know $\Omega_{\mathbb{R}} = \sqrt{\frac{2\pi}{\sqrt{3}}}$ (cf.\cite{otsubo}), hence we obtain the theorem.
\end{proof}

\subsection{Conductor 32}

The formula for $r_{\mathscr{D}}(e_{E_{32}})$ due to Otsubo is as follows 
(see \cite{otsubo}, Section 5.2 for the formula of $f^{*}\omega_{E_{32}}$).
\begin{theo}[{\cite[Theorem 3.2]{otsubo}}]\label{n=4noregfor}
Let $e_{E_{32}} \in H^{2}_{\mathscr{M}}(E_{32}, \mathbb{Q}(2))_{\mathbb{Z}}$ be the element defined in Section \ref{regulator}. 
Then we have
\begin{equation*}
r_{\mathscr{D}}(e_{E_{32}}) = -\frac{\sqrt{2}}{16\sqrt{\pi}}\left( \tilde{F}\left(\frac{1}{4}, \frac{1}{2} \right) 
- \tilde{F}\left(\frac{3}{4},\frac{1}{2}\right) \right) \left( \omega_{E_{32}} - \overline{\omega_{E_{32}}} \right)
\end{equation*}
and $\tilde{F}\left(\frac{1}{4}, \frac{1}{2} \right) - \tilde{F}\left(\frac{3}{4},\frac{1}{2}\right) \neq 0$.
\end{theo}

By comparing the formula above with Theorem \ref{n=4nol}, we prove the Beilinson conjectures for $E_{32}$.
\begin{theo}\label{comp2}
With notations as above, we have
\begin{equation*}
r_{\mathscr{D}}(e_{E_{32}}) = - 
\frac{1}{2} L^{\prime}(E_{32},0)\Omega_{\mathbb{R}}(\omega_{E_{32}} - \overline{\omega_{E_{32}}}).
\end{equation*}
\end{theo}
\begin{proof}
By \eqref{dixon}, we have 
\begin{align*}
\tilde{F}\left( \frac{1}{4}, \frac{1}{2} \right)
&=
\sqrt{\frac{2}{\pi}}\Gamma^{2}\left(\frac{1}{4}\right)
{_{3}F_{2}}\left[ 
\left.
\begin{matrix}
\frac{1}{2}, \frac{1}{2}, 1 \\
\frac{3}{2}, \frac{3}{4}
\end{matrix}
\right| 1
\right], \\
\tilde{F}\left(\frac{3}{4} , \frac{1}{2}\right)
&
=4\sqrt{\frac{2}{\pi}}\Gamma^{2}\left(\frac{3}{4}\right)
{_{3}F_{2}}\left[ 
\left.
\begin{matrix}
\frac{1}{2}, \frac{1}{2}, 1 \\
\frac{3}{2}, \frac{5}{4}
\end{matrix}
\right| 1
\right].
\end{align*}
Hence, by Theorems \ref{n=4nol} and \ref{n=4noregfor}, we have
\begin{equation*}
r_{\mathscr{D}}(e_{E_{32}})
= -\frac{4\sqrt{2}}{\pi\sqrt{\pi}} L(E_{32}, 2)\left(\omega_{E_{32}} - \overline{\omega_{E_{32}}} \right) .
\end{equation*}
By the functional equation \eqref{functional} and the fact that the root number is $1$, we have
\begin{align*}
r_{\mathscr{D}}(e_{E_{32}})
= -\frac{\sqrt{2\pi}}{2} L^{\prime}(E_{32}, 0)\left( \omega_{E_{32}} - \overline{\omega_{E_{32}}} \right) .
\end{align*}
We know $\Omega_{\mathbb{R}} = \sqrt{2\pi}$ (cf.\cite{otsubo}), hence we obtain the theorem.
\end{proof}

\subsection{Conductor 64}

The formula for $r_{\mathscr{D}}(e_{E_{64}})$ due to Otsubo is as follows. 
It is not difficult to see that $f^{*}\omega_{E_{64}}$ is proportional to $\widetilde{\omega}_{4}^{1,1}$ 
(see \cite[Section 3.2]{otsubo} for the notation). Then similarly as in loc. cit., we obtain 
$f^{*}\omega_{E_{64}} = \frac{\sqrt{\pi}}{2}\widetilde{\omega}_{4}^{1,1}$.
\begin{theo}[{\cite[Theorem 3.2]{otsubo}}]\label{n=4noregfor2}
Let $e_{E_{64}} \in H^{2}_{\mathscr{M}}(E_{64}, \mathbb{Q}(2))_{\mathbb{Z}}$ be the element defined in Section \ref{regulator}. 
Then we have
\begin{equation*}
r_{\mathscr{D}}(e_{E_{64}}) = -\frac{1}{16\sqrt{\pi}}\left( \tilde{F}\left(\frac{1}{4}, \frac{1}{4} \right) 
- \tilde{F}\left(\frac{3}{4},\frac{3}{4}\right) \right) \left( \omega_{E_{64}} - \overline{\omega_{E_{64}}} \right)
\end{equation*}
and $\tilde{F}\left(\frac{1}{4}, \frac{1}{4} \right) - \tilde{F}\left(\frac{3}{4},\frac{3}{4}\right) \neq 0$. 
\end{theo}

By comparing the formula above with Theorem \ref{cond64}, we prove the Beilinson conjectures for $E_{64}$.
\begin{theo}\label{comp3}
With notations as above, we have
\begin{equation*}
r_{\mathscr{D}}(e_{E_{64}}) = - 
\frac{1}{2} L^{\prime}(E_{64},0)\Omega_{\mathbb{R}}(\omega_{E_{64}} - \overline{\omega_{E_{64}}}).
\end{equation*}
\end{theo}
\begin{proof}
By \eqref{dixon}, we have 
\begin{align*}
\tilde{F}\left( \frac{1}{4}, \frac{1}{4} \right)
&=
\frac{4\Gamma^{2}\left(\frac{1}{4}\right)}{\sqrt{\pi}}
{_{3}F_{2}}\left[ 
\left.
\begin{matrix}
\frac{1}{4}, \frac{1}{4}, 1 \\
\frac{1}{2}, \frac{5}{4}
\end{matrix}
\right| 1
\right], \\
\tilde{F}\left(\frac{3}{4} , \frac{3}{4}\right)
&
=\frac{8\Gamma^{2}\left(\frac{3}{4}\right)}{3\sqrt{\pi}}
{_{3}F_{2}}\left[ 
\left.
\begin{matrix}
\frac{3}{4}, \frac{3}{4}, 1 \\
\frac{3}{2}, \frac{7}{4}
\end{matrix}
\right| 1
\right].
\end{align*}
Hence, by Theorems \ref{cond64} and \ref{n=4noregfor2}, we have
\begin{equation*}
r_{\mathscr{D}}(e_{E_{64}})
= -\frac{8}{\pi\sqrt{\pi}} L(E_{64}, 2)\left(\omega_{E_{64}} - \overline{\omega_{E_{64}}} \right) .
\end{equation*}
By the functional equation \eqref{functional} and the fact that 
the root number is $1$ (cf.\cite[p.84, Theorem]{kob}), we have
\begin{align*}
r_{\mathscr{D}}(e_{E_{64}})
= -\frac{\sqrt{\pi}}{2} L^{\prime}(E_{64}, 0)\left( \omega_{E_{64}} - \overline{\omega_{E_{64}}} \right) .
\end{align*}
We know $\Omega_{\mathbb{R}} = \sqrt{\pi}$ (cf.\cite{otsubo}), hence we obtain the theorem.
\end{proof}

\section*{Acknowledgment}
This paper is based on the author's master's thesis at Chiba University.
I am very grateful to Noriyuki Otsubo and Shigeki Matsuda for valuable advises.
I would like to thank Mathew Rogers and Wadim Zudilin for helpful comments.

\end{document}